\documentclass[leqno,11pt]{amsart}%
\usepackage{amsmath,amsfonts,amssymb,amsthm}
\usepackage{graphicx}
\usepackage{amsmath}
\usepackage{amsfonts}
\usepackage{amssymb}
\usepackage{graphicx}%
\providecommand{\U}[1]{\protect\rule{.1in}{.1in}}
\newtheorem{theorem}{Theorem}

\newtheorem{corollary}[theorem]{Corollary}

\newtheorem{lemma}[theorem]{Lemma}

\newtheorem{proposition}[theorem]{Proposition}

\begin{document}
\title{Group invariant solutions of certain partial differential equations}
\author{Jaime Ripoll}
\author{Friedrich Tomi}
\date{}
\subjclass[2000]{53A10, 53C42, 49Q05, 49Q20}
\keywords{elliptic pdes, group invariant solutions }

\begin{abstract}
Let $M$ be a complete Riemannian manifold and $G$ a Lie subgroup of the
isometry group of $M$ acting freely and properly on $M.$ We study the
Dirichlet Problem%
\[
\left\{
\begin{array}
[c]{l}%
\operatorname{div}\left(  \frac{a\left(  \left\Vert \nabla u\right\Vert
\right)  }{\left\Vert \nabla u\right\Vert }\nabla u\right)  =0\text{ in
}\Omega\\
u|\partial\Omega=\varphi
\end{array}
\right.
\]
where $\Omega$ is a $G-$invariant domain of $C^{2,\alpha}$ class in $M$ and
$\varphi\in C^{2,\alpha}\left(  \partial\overline{\Omega}\right)  $ a
$G-$invariant function. Two classical PDE's are included in this family: the
$p-$Laplacian $(a(s)=s^{p-1},$ $p>1)$ and the minimal surface equation
$(a(s)=s/\sqrt{1+s^{2}}).$ Our motivation, by using the concept of Riemannian
submersion, is to present a method in studying $G$-invariant solutions for
noncompact Lie groups which allows the reduction of the Dirichlet problem on
unbounded domains to one on bounded domains.

\end{abstract}
\maketitle

\section{Introduction\label{intro}}

\qquad Let $M$ be a complete Riemannian manifold and $G$ a Lie subgroup of the
isometry group of $M$ acting freely and properly on $M.$ In this paper we
study the Dirichlet Problem (DP)%
\begin{equation}
\left\{
\begin{array}
[c]{l}%
\operatorname{div}\left(  \frac{a\left(  \left\Vert \nabla u\right\Vert
\right)  }{\left\Vert \nabla u\right\Vert }\nabla u\right)  =0\text{ in
}\Omega\\
u|\partial\Omega=\varphi
\end{array}
\right.  \label{PDE1}%
\end{equation}
where $\Omega$ is a $G-$invariant domain of $C^{2,\alpha}$ class in $M$ with
non empty boundary and $\varphi\in C^{2,\alpha}\left(  \partial\overline
{\Omega}\right)  $ a $G-$invariant function. We require, as minimal
conditions, that
\[
a\in C^{0}\left(  \left[  0,\infty\right[  \right)  \cap C^{1}\left(  \left]
0,\infty\right[  \right)  ,\text{ }a(s)>0,\text{ }a^{\prime}(s)>0
\]
for $s>0.$ Two classical PDE's are included in this family: the $p-$Laplacian
$(a(s)=s^{p-1},$ $p>1)$ and the minimal surface equation $(a(s)=s/\sqrt
{1+s^{2}}).$ The main motivation in studying problem (\ref{PDE1}) is to
reduce, by considering non compact Lie groups, the DP in unbounded\ domains of
$M$ to bounded domains in the quotient space $M/G$. One should mention that
the existence and uniqueness of solutions to the Dirichlet problem for the
minimal surface equation on unbounded domains has a history which goes back to
the investigations of J. C. C. Nitsche on the so called exterior Dirichlet
problem (see \cite{N}). Several authors continued Nitsche's study in the
Euclidean space \cite{Kr}, \cite{KT}, \cite{Ku}, \cite{RT2} and in Riemannian
spaces in \cite{ER}. Existence and uniqueness for more general unbounded
domains were studied in \cite{ERo}, \cite{CK}, \cite{Sp} in the Euclidean
space and in \cite{T} in the Riemannian setting.

To state our main results we need to remark some facts and to introduce some
terminology. We first observe that the assumptions on the action of $G$ on $M$
guarantee that the orbit space $M/G=\left\{  G(p)\text{
$\vert$
}p\in M\right\}  ,$ where $G(p)=\left\{  g(p)\text{
$\vert$
}g\in G\right\}  ,$ $p\in M,$ is a differentiable manifold with respect to the
quotient topology, and the projection $\pi:M\rightarrow M/G$ is a submersion.
Since the action of $G$ on $M$ is by isometries we may, and we will, consider
in $M/G$ the Riemannian metric such that $\pi$ becomes a Riemannian submersion.

Regarding the PDE in (\ref{PDE1}), the $p-$Laplace and the minimal surface
equation are representatives of two classes of PDE's which are distinguished
as follows (see also \cite{RT1}). The PDE in (\ref{PDE1}) may be written in
the equivalent form%
\begin{equation}
\left\Vert \nabla u\right\Vert ^{2}\Delta u+\left(  \frac{\left\Vert \nabla
u\right\Vert a^{\prime}(\left\Vert \nabla u\right\Vert )}{a(\left\Vert \nabla
u\right\Vert )}-1\right)  \nabla^{2}u\left(  \nabla u,\nabla u\right)
=0\label{PDE2}%
\end{equation}
where $\Delta$ and $\nabla^{2}$ denote the Laplacian and the Hessian. The
quadratic form%
\begin{equation}
q\left(  \xi,\xi\right)  =\left\Vert \nabla u\right\Vert ^{2}\left\vert
\xi\right\vert ^{2}+b\left(  \left\Vert \nabla u\right\Vert \right)
\left\langle \xi,\nabla u\right\rangle ^{2}\label{qf}%
\end{equation}
associated with (\ref{PDE2}), where
\begin{equation}
b\left(  s\right)  =\frac{sa^{\prime}(s)}{a(s)}-1,\label{be}%
\end{equation}
has the eigenvalue%
\[
\beta=\left\Vert \nabla u\right\Vert ^{2}\left(  1+b\left(  \left\Vert \nabla
u\right\Vert \right)  \right)
\]
in the direction of $\nabla u$ and the maximal eigenvalue%
\[
\Lambda=\left\Vert \nabla u\right\Vert ^{2}\max\left\{  1,1+b\left(
\left\Vert \nabla u\right\Vert \right)  \right\}  .
\]
We may easily see that
\[
\frac{\beta}{\Lambda}=1+b^{-}%
\]
where $b^{-}=\min\left\{  b,0\right\}  .$ We consider the following two
possibilities, which include, respectively, the $p-$Laplacian and the minimal
surface equation:

\begin{itemize}
\item \textbf{Condition I }\emph{Mild decay of the eigenvalue ratio}:%
\[
\left(  1+b^{-}(s)\right)  s^{2}\geq g(s),\text{ }s\geq s_{0}>0
\]
where $g$ is non decreasing and%
\[
\int_{s_{0}}^{\infty}\frac{g(s)}{s^{2}}ds=+\infty
\]

\item \textbf{Condition II }\emph{Strong decay of the eingenvalue ratio:}%
\[
\left(  1+b^{-}(s)\right)  s^{2}\geq g(s),\text{ }s\geq s_{0}>0
\]
where $g$ is non increasing and%
\[
\int_{s_{0}}^{\infty}\frac{g(s)}{s}ds=+\infty.
\]

\end{itemize}

The MDER case was introduced by James Serrin in \cite{S} as \emph{regularly
elliptic equations. }We observe that in both MDER and SDER classes the
equations can be singular or degenerated. A typical example, which occurs in
the MDER class, is the $p-$Laplacian PDE (it is singular if $1<p<2$ and
degenerated if $p>2).$

In this paper we consider the same class of equations as in \cite{RT1},
namely: Writing%
\[
a(s)=s^{p-1}A(s),\text{ }s\geq0,
\]
for some $p>1$ we require that
\begin{equation}
A\in C^{2}\left(  \left[  0,\infty\right[  \right)  ,\text{ }A(s)>0\text{ for
}s\geq0 \label{pltype}%
\end{equation}
and that%

\begin{equation}
\min\left\{  1,p-1\right\}  +\frac{sA^{\prime}(s)}{A(s)}>0 \label{nmin1}%
\end{equation}
for all $s\geq0.$ Note that (\ref{nmin1}) implies that $a^{\prime}(s)>0$ for
$s>0.$

We call $p=2$ the \emph{regular} case to which the classical theory of
elliptic differential equations can be applied. The study of the DP in the
nonregular case $p\neq2$ can be reduced to the regular one by a perturbation
technique \cite{RT1}. The minimal surface equation is regular but not the
$p-$Laplacian if $p\neq2,$ as it is easy to see.

Depending on the case which is being investigated extra conditions have to be
required on the PDE. They are:

\begin{itemize}
\item \textbf{Condition III }There are $\beta>0$ and a function $h:\left[
0,\infty\right[  \rightarrow\mathbb{R}$ with $h(s)\rightarrow\infty$ $\left(
s\rightarrow\infty\right)  $ such that%
\[
\left(  b(s)+1-\beta b^{\prime}(s)^{+}s\right)  s^{2}\geq h(s),\text{ }%
s\in\left[  0,\infty\right[  .
\]

\item \textbf{Condition IV }there are positive numbers $\alpha$ and $s_{0}$
such that%
\[
\left(  -b^{\prime}(s)s-\left(  b(s)+1\right)  \right)  s^{2}\geq\alpha,\text{
}s\geq s_{0}%
\]

\end{itemize}

To state our first result in the MDER case, for smooth boundary data, we
recall that $u\in C^{1}\left(  \Omega\right)  \cap$ $C^{0}\left(
\overline{\Omega}\right)  $ is a weak solution of (\ref{PDE1}) if
$u|\partial\Omega=\varphi$ and%
\[
\int_{\Omega}\frac{a\left(  \left\Vert \nabla u\right\Vert \right)
}{\left\Vert \nabla u\right\Vert }\left\langle \nabla u,\nabla g\right\rangle
\omega=0
\]
for all $g\in C^{\infty}\left(  \overline{\Omega}\right)  $ with compact
support in $\Omega,$ where $\omega$ is the volume form of $M.$ If the PDE is
regular then regularity theory implies that a weak solution $u$ is in
$C^{2,\alpha}\left(  \overline{\Omega}\right)  $ and satisfies (\ref{PDE1}) in
the classical sense.

\begin{theorem}
[{\small the MDER case for smooth boundary data}]\label{mder1}Let $M$ be a
complete Riemannian manifold and $G$ a Lie subgroup of the isometry group of
$M$ acting freely and properly on $M.$ Let $\Omega$ be a $G-$invariant domain
of $C^{2,\alpha}$ class. Assume that $\pi\left(  \Omega\right)  $ is bounded
in $M/G\ $and that Conditions I and III are satisfied. Then the Dirichlet
problem (\ref{PDE1}) has an unique $G-$invariant weak solution $u\in
C^{1}\left(  \overline{\Omega}\right)  $ for any $G-$invariant boundary data
$\varphi$ of $C^{2,\alpha}$ class$.$ If $p=2$ the solution is of $C^{2,\alpha
}\ $class in $\overline{\Omega}$ and hence a classical solution.
\end{theorem}

Notice\ that Conditions I and III are fulfilled if $b(s)+1=cs^{m}$ for some
$c>0,$ $m\geq0.$ In particular, Theorem \ref{mder1} holds for the $p-$Laplace
equation where $b(s)=p-2$ ($p>1).$

Differently of the mild decay eigenvalue ratio, the solvability of the
Dirichlet problem in the strong decay case requires the usual mean convexity
of the domain.

\begin{theorem}
[{\small the SDER case for smooth boundary data}]\label{sder1}Let $M$ be a
complete Riemannian manifold and let $G$ be a Lie subgroup of the isometry
group of $M$ acting freely and properly on $M.$ Let $\Omega\subset M$ be a
$G-$invariant $C^{2,\alpha}$ domain with bounded projection on $M/G.$ Assume
that Conditions II and IV are satisfied and that the mean curvature of
$\partial\Omega$ with respect to the interior normal vector of $\partial
\Omega$ as well as of the inner parallel hypersurfaces of $\partial\Omega$ in
some neighborhood of $\partial\Omega$ is nonnegative. Then the Dirichlet
problem (\ref{PDE1}) has an unique weak $G-$invariant solution for any
$G-$invariant boundary data $\varphi\in C^{2,\alpha}\left(  \partial
\Omega\right)  .$ If $p=2$ the solution is of $C^{2,\alpha}\ $class in
$\overline{\Omega}$ and hence a classical solution.
\end{theorem}

Since $\Omega$ is invariant by $G$ and has compact projection on $M/G$, the
mean convexity of the parallel hypersurfaces of $\partial\Omega$ holds in a
neighborhood of $\partial\Omega$ if one requires $\partial\Omega$ to be
strictly mean convex$.$ It also holds if $\partial\Omega$ is only mean convex
and $M$ has nonnegative Ricci curvature in an uniform neighborhood of
$\partial\Omega.$

We conclude the paper with two examples to illustrate the power of Theorem
\ref{sder1}: In the first one we consider the Dirichlet problem for the
minimal surface equation on $\mathbb{R}^{3}$ which invariance under helicoidal
motions$,$ in the second one the asymptotic Dirichlet problem on
$\mathbb{H}^{n}$ with invariance under transvections.

\section{A new PDE}

\qquad Let $M$ be a complete Riemannian manifold and let $G$ be a Lie subgroup
of the isometry group of $M.$ Assume that $G$ acts freely and properly on $M.$
Given $p\in M$ let $G(p)=\left\{  g(p)\text{
$\vert$
}g\in G\right\}  $ be the orbit of $G$ through $p.$ Then the orbit space
\[
M/G:=\left\{  G(p)\text{
$\vert$
}p\in M\right\}
\]
is a differentiable manifold with the quotient topology and the projection
$\pi:M\rightarrow M/G$ is a submersion. We consider in $M/G$ the Riemannian
metric such that $\pi$ becomes a Riemannian submersion.

We denote by $\overrightarrow{H}_{G}$ the mean curvature vector of the orbits
of $G$ that is,%
\[
\overrightarrow{H}_{G}=\sum_{i=1}^{k}\left(  \nabla_{E_{i}}E_{i}\right)
^{\bot},
\]
where $\left\{  E_{i}\right\}  $ is a local orthonormal frame tangent to a
orbit of $G.$ Note that $\overrightarrow{H}_{G}$ is $G-$invariant, $g_{\ast
}\overrightarrow{H}_{G}=\overrightarrow{H}_{G}\circ g$ for all $g\in G.$ Then
it projects into a vector field of $M/G$ which we denote by $J.$ We denote by
$\nabla$ the gradient and also the Riemannian connections in $M$ and $M/G$.
The meaning of the notation will be clear~from the context.

If $X$ is a vector field in $M/G$ we denote by $\overline{X}$ the vector field
in $M$ determined by the horizontal lift of $X$ to $M$ namely, $\pi_{\ast
}\overline{X}=X\circ\pi$ and $\overline{X}\left(  p\right)  \in T_{p}G\left(
p\right)  ^{\bot}$ for all $p\in M.$

\begin{proposition}
\label{ab}Let $\Omega$ be a $G-$invariant domain of $C^{2}$ class in $M$,
$u\in C^{2}\left(  \Omega\right)  \cap C^{0}\left(  \overline{\Omega}\right)
$ and $\varphi\in C^{0}\left(  \partial\Omega\right)  $ $G-$invariant
functions, $\Lambda=\pi(\Omega)\subset M/G$, $v\in C^{2}\left(  \Lambda
\right)  \cap C^{0}\left(  \overline{\Lambda}\right)  $ and $\psi\in
C^{0}\left(  \partial\Lambda\right)  $ such that $u=v\circ\pi,$ $\varphi
=\psi\circ\pi$. Then%
\[
\frac{a\left(  \left\Vert \nabla u\right\Vert \right)  }{\left\Vert \nabla
u\right\Vert }=\frac{a\left(  \left\Vert \nabla v\right\Vert \right)
}{\left\Vert \nabla v\right\Vert }\circ\pi
\]
and $u$ is a solution of the DP%
\[
\left\{
\begin{array}
[c]{l}%
\operatorname{div}_{M}\left(  \frac{a\left(  \left\Vert \nabla u\right\Vert
\right)  }{\left\Vert \nabla u\right\Vert }\nabla u=0\right)  \text{ in
}\Omega\\
u|\partial\Omega=\varphi
\end{array}
\right.
\]
if and only if $v$ is a solution of the DP%
\begin{equation}
\left\{
\begin{array}
[c]{l}%
\operatorname{div}_{M/G}\left(  \frac{a\left(  \left\Vert \nabla v\right\Vert
\right)  }{\left\Vert \nabla v\right\Vert }\nabla v\right)  -\frac{a\left(
\left\Vert \nabla v\right\Vert \right)  }{\left\Vert \nabla v\right\Vert
}\left\langle \nabla v,J\right\rangle =0\text{ in }\Lambda\\
v|\partial\Lambda=\psi.
\end{array}
\right.  \label{npde}%
\end{equation}

\end{proposition}

\begin{proof}
Assume that $n=\dim M$ and $k=\dim G.$ Given $p\in M$ let $E_{1},...,E_{m}$ be
a local orthonormal frame in a neighborhood of $p$ which is orthogonal to the
orbits of $G,$ $m=n-k,$ and let $X_{1},...,X_{k}$ be Killing fields determined
by $G$ which are orthonormal at $p.$

Using that $E_{i}=\overline{F_{i}}$ where $F_{i},$ $i=1,...,n-1,$ is an local
orthonormal frame around $\pi(p)$ in $M/G$ we obtain $\nabla u=\overline
{\nabla v}.$ It follows that the function $a\left(  \left\Vert \nabla
u\right\Vert \right)  /\left\Vert \nabla u\right\Vert $ is invariant by $G$
and then it is well defined in $M/G$ and we have%
\[
\frac{a\left(  \left\Vert \nabla u\right\Vert \right)  }{\left\Vert \nabla
u\right\Vert }=\frac{a\left(  \left\Vert \nabla v\right\Vert \right)
}{\left\Vert \nabla v\right\Vert }\circ\pi.
\]

Moreover, at $p,$%
\[
\overrightarrow{H}_{G}(p)=\sum_{i=1}^{k}\left(  \nabla_{X_{i}}X_{i}\right)
^{\bot}(p)
\]
and
\begin{align*}
\operatorname{div}_{M}\left(  \frac{a\left(  \left\Vert \nabla u\right\Vert
\right)  }{\left\Vert \nabla u\right\Vert }\nabla u\right)   &  =\sum
_{i=1}^{m}\left\langle \nabla_{E_{i}}\left(  \frac{a\left(  \left\Vert \nabla
u\right\Vert \right)  }{\left\Vert \nabla u\right\Vert }\nabla u\right)
,E_{i}\right\rangle \\
&  +\sum_{i=1}^{k}\left\langle \nabla_{X_{i}}\left(  \frac{a\left(  \left\Vert
\nabla u\right\Vert \right)  }{\left\Vert \nabla u\right\Vert }\nabla
u\right)  ,X_{i}\right\rangle .
\end{align*}

Also,%
\begin{align*}
\left\langle \nabla_{E_{i}}\left(  \frac{a\left(  \left\Vert \nabla
u\right\Vert \right)  }{\left\Vert \nabla u\right\Vert }\overline{\nabla
}u\right)  ,E_{i}\right\rangle  &  =E_{i}\left(  \frac{a\left(  \left\Vert
\nabla u\right\Vert \right)  }{\left\Vert \nabla u\right\Vert }\right)
\left\langle \nabla u,E_{i}\right\rangle \\
&  +\frac{a\left(  \left\Vert \nabla u\right\Vert \right)  }{\left\Vert \nabla
u\right\Vert }\left\langle \nabla_{E_{i}}\nabla u,E_{i}\right\rangle \\
&  =F_{i}\left(  \frac{a\left(  \left\Vert \nabla v\right\Vert \right)
}{\left\Vert \nabla v\right\Vert }\right)  \left\langle \nabla v,F_{i}%
\right\rangle \circ\pi\\
&  +\frac{a\left(  \left\Vert \nabla v\right\Vert \right)  }{\left\Vert \nabla
v\right\Vert }\circ\pi\left\langle \nabla_{\overline{F}_{i}}\nabla
u,\overline{F}_{i}\right\rangle .
\end{align*}

Using O'Neal's formula for Riemannian submersions (\cite{doC}, exercises of
Ch. 8) namely,%
\begin{equation}
\nabla_{\overline{X}}\overline{Y}=\overline{\nabla_{X}Y}+\frac{1}{2}\left[
\overline{X},\overline{Y}\right]  ^{V}, \label{ers}%
\end{equation}
where $X,Y$ are vector fields on $M/G,$ $\overline{X}$ and $\overline{Y}$
their horizontal lift in $M$ and $V$ the orthogonal projection on $TG,$ we
obtain%
\[
\left\langle \nabla_{\overline{F}_{i}}\nabla u,\overline{F}_{i}\right\rangle
=\left\langle \nabla_{\overline{F}_{i}}\overline{\nabla v},\overline{F}%
_{i}\right\rangle =\left\langle \overline{\nabla_{F_{i}}\nabla v},\overline
{F}_{i}\right\rangle =\left\langle \nabla_{F_{i}}\nabla v,F_{i}\right\rangle
\circ\pi
\]
and hence%
\begin{align*}
\sum_{i=1}^{m}\left\langle \nabla_{E_{i}}\left(  \frac{a\left(  \left\Vert
\nabla u\right\Vert \right)  }{\left\Vert \nabla u\right\Vert }\nabla
u\right)  ,E_{i}\right\rangle  &  =\sum_{i=1}^{m}\left\{  F_{i}\left(
\frac{a\left(  \left\Vert \nabla v\right\Vert \right)  }{\left\Vert \nabla
v\right\Vert }\right)  \left\langle \nabla v,F_{i}\right\rangle \circ
\pi\right. \\
&  +\left.  \left(  \frac{a\left(  \left\Vert \nabla v\right\Vert \right)
}{\left\Vert \nabla v\right\Vert }\sum_{i=1}^{m}\left\langle \nabla_{F_{i}%
}\nabla v,F_{i}\right\rangle \right)  \circ\pi\right\} \\
&  =\operatorname{div}_{M/G}\left(  \frac{a\left(  \left\Vert \nabla
v\right\Vert \right)  }{\left\Vert \nabla v\right\Vert }\nabla v\right)
\circ\pi.
\end{align*}

Since $\left\langle \nabla u,X_{i}\right\rangle =0$, $1\leq i\leq k,$ it
follows that, at $p$,%
\begin{align*}
\sum_{i=1}^{k}\left\langle \nabla_{X_{i}}\left(  \frac{a\left(  \left\Vert
\nabla u\right\Vert \right)  }{\left\Vert \nabla u\right\Vert }\nabla
u\right)  ,X_{i}\right\rangle  &  =\frac{a\left(  \left\Vert \nabla
u\right\Vert \right)  }{\left\Vert \nabla u\right\Vert }\sum_{i=1}%
^{k}\left\langle \nabla_{X_{i}}\nabla u,X_{i}\right\rangle \\
&  =-\frac{a\left(  \left\Vert \nabla u\right\Vert \right)  }{\left\Vert
\nabla u\right\Vert }\sum_{i=1}^{k}\left\langle \nabla u,\nabla_{X_{i}}%
X_{i}\right\rangle \\
&  =-\left[  \frac{a\left(  \left\Vert \nabla v\right\Vert \right)
}{\left\Vert \nabla v\right\Vert }\left\langle \nabla v,J\right\rangle
\right]  \circ\pi
\end{align*}
which concludes the proof of the proposition.
\end{proof}

We use the notations of the proposition above to prove a lemma to be used in
the proof of Theorem \ref{sder1}:

\begin{lemma}
\label{hahb}Under the same hypothesis of Proposition \ref{ab}, denote by
$H_{\partial\Omega}$ and $H_{\partial\Lambda}$ the non normalized mean
curvature of $\partial\Omega$ and $\partial\Lambda$ with respect to the unit
normal vector fields pointing to the interior of the domains. Then%
\[
H_{\partial\Omega}=H_{\partial\Lambda}\circ\pi+\left\langle \overrightarrow
{H}_{G},\nu\right\rangle
\]
where $\nu$ is the unit normal vector field along $\partial\Omega$ pointing to
$\Omega.$
\end{lemma}

\begin{proof}
Let $p\in\partial\Omega$ be given. Let $F_{1},...,F_{m-1}$ be a local
orthonormal frame in a neighborhood of $\pi\left(  p\right)  ,$ tangent to
$\partial\Lambda.$ Let $E_{i}=\overline{F}_{i}$ be the horizontal lift of
$F_{i}$ to $M$. Let $X_{1},...,X_{k}$ be Killing fields determined by $G$
which are orthonormal at $p.$ Let $\nu$ be the unit normal vector field
orthogonal to $\partial\Omega$ pointing to $\Omega.$ Since clearly $\nu$ is
invariant by $G$ and horizontal, it projects into the unit normal vector
$\eta$ orthogonal to $\partial\Lambda$ pointing to $\Lambda.$ We then have, at
$p,$ using (\ref{ers}),%
\begin{align*}
H_{\partial\Omega}  &  =\sum_{i=1}^{m-1}\left\langle \nabla_{E_{i}}E_{i}%
,\nu\right\rangle +\sum_{i=1}^{k}\left\langle \nabla_{X_{i}}X_{i}%
,\nu\right\rangle \\
&  =\sum_{i=1}^{m-1}\left\langle \nabla_{\overline{F}_{i}}\overline{F}%
_{i},\overline{\eta}\right\rangle +\sum_{i=1}^{k}\left\langle \left(
\nabla_{X_{i}}X_{i}\right)  ^{\bot},\nu\right\rangle \\
&  =\sum_{i=1}^{m-1}\left\langle \nabla_{F_{i}}F_{i},\eta\right\rangle
\circ\pi+\sum_{i=1}^{k}\left\langle \overrightarrow{H}_{G},\nu\right\rangle \\
&  =H_{\partial\Lambda}\circ\pi+\sum_{i=1}^{k}\left\langle \overrightarrow
{H}_{G},\nu\right\rangle
\end{align*}
proving the lemma.
\end{proof}

\subsection{Remark}

\qquad It is easy to see that if $\Omega$ is a bounded $C^{1}$ domain then the
PDE in (\ref{PDE1}) is the Euler-Lagrange equation of the functional%
\[
F\left(  u\right)  =\int_{\Omega}\phi\left(  \left\Vert \nabla u\right\Vert
\right)  ,\text{ }u\in C^{1}\left(  \overline{\Omega}\right)  ,
\]
where $\phi^{\prime}=a.$

Let us now assume that $G$ is compact, $\Omega$ a $G-$invariant domain and
$u\in C^{1}\left(  \overline{\Omega}\right)  $ a $G-$invariant function. Then,
setting $\Lambda=\pi\left(  \Omega\right)  ,$ defining $v\in C^{1}\left(
\overline{\Lambda}\right)  $ such that $u=v\circ\pi$ and denoting by
$\mathcal{H}^{s}$ the $s-$dimensional Hausdorff measure of $M$ we have,\ by
the coarea formula,%
\begin{align*}
F\left(  u\right)   &  =\int_{x\in\Lambda}\left(  \int_{\pi^{-1}\left(
x\right)  }\frac{1}{\left\Vert \operatorname*{Jac}\left(  \pi\right)
\right\Vert }\phi\left(  \left\Vert \nabla u\right\Vert \right)
d\mathcal{H}^{k}\right)  d\mathcal{H}^{m}\\
&  =\int_{x\in\Lambda}\left(  \int_{\pi^{-1}\left(  x\right)  }\phi\left(
\left\Vert \overline{\nabla v}\right\Vert \right)  d\mathcal{H}^{k}\right)
d\mathcal{H}^{m}\\
&  =\int_{x\in\Lambda}\phi\left(  \left\Vert \nabla v\left(  x\right)
\right\Vert \right)  \left(  \int_{\pi^{-1}\left(  x\right)  }d\mathcal{H}%
^{k}\right)  d\mathcal{H}^{m}\\
&  =\int_{x\in\Lambda}\phi\left(  \left\Vert \nabla v\left(  x\right)
\right\Vert \right)  \operatorname*{Vol}\left(  \pi^{-1}\left(  x\right)
\right)  d\mathcal{H}^{m}.
\end{align*}
where we have used that $\left\Vert \operatorname*{Jac}\left(  \pi\right)
\right\Vert =1$ as it is easy to see. The last integral defines a functional
on $M/G$ namely%
\[
F_{G}\left(  v\right)  =\int_{\Lambda}V\phi\left(  \left\Vert \nabla
v\right\Vert \right)  d\mathcal{H}^{m},\text{ }v\in C^{1}\left(
\overline{\Lambda}\right)  ,
\]
where $V\in C^{1}\left(  \overline{\Lambda}\right)  $ is defined by $V\left(
x\right)  =$ $\operatorname*{Vol}\left(  \pi^{-1}\left(  x\right)  \right)  ,$
$x\in\overline{\Lambda}$, and it follows that the new PDE (the PDE in
(\ref{npde})) is the Euler-Lagrange equation of $F_{G}.$

\section{\label{thepde}Proofs of Theorems \ref{mder1} and \ref{sder1}}

\qquad In this section we prove that the conditions for the solvability of the
DP
\begin{equation}
\left\{
\begin{array}
[c]{l}%
\operatorname{div}_{N}\left(  \frac{a\left(  \left\Vert \nabla u\right\Vert
\right)  }{\left\Vert \nabla u\right\Vert }\nabla u\right)  -\frac{a\left(
\left\Vert \nabla u\right\Vert \right)  }{\left\Vert \nabla u\right\Vert
}\left\langle \nabla u,J\right\rangle =0\text{ in }\Lambda\\
u|\partial\Lambda=\psi,
\end{array}
\right.  \label{pdej}%
\end{equation}
where $N$ is a complete Riemannian manifold, $\Lambda$ a bounded domain of
$C^{2,\alpha}$ class in $N$ and $J$ is a smooth vector field on $N$ are the
same as in the case that $J=0$ with exception to the boundary condition as we
explain later. We deal only with the regular case ($p=2).$ The case of
nonregular equations equations are dealt with by the same perturbation
technique used in \cite{RT1}. For proving Theorems \ref{mder1} and \ref{sder1}
we then take $N=M/G,$ $\Lambda=\pi\left(  \Omega\right)  $, $\psi$ such that
$\varphi\circ\pi=\psi$ and apply Proposition \ref{ab}$.$

We begin observing that the terminologies \emph{regular, MDER }(Condition I)
and\emph{ SDER} (Condition II) as well as Conditions III and IV introduced
previously apply to the in PDE (\ref{pdej}) since they depend only on the
behavior of the function $a.$

Since the equation (\ref{pdej}) is not of divergence form the comparison
principle for weak solutions as in \cite{RT1} is not immediately applicable.
Instead, we may use the classical maximum and minimum principles for sub and
supersolutions of $C^{2}$ class, respectively (\cite{GT}, Chapter 3). On
account of our assumption that $A$ is in $C^{2}\left(  \left[  0,\infty
\right[  \right)  $ one easily sees that if $v,w\in C^{2}\left(
\Lambda\right)  $ are sub and supersolutions of (\ref{pdej}) then $w-v$
satisfies an elliptic linear differential inequality with locally bounded
coefficients to which the maximum principle applies. Hence, if%
\[
\liminf_{x\rightarrow\partial\Lambda}\left(  w-v\right)  (x)\geq0
\]
then $w-v\geq0$ in $\Lambda$ that is, the PDE (\ref{pdej}) satisfies the
comparison principle. In particular, since the constant functions are
solutions of (\ref{pdej}), it follows that if $u\in C^{2}\left(
\Lambda\right)  \cap C^{0}\left(  \overline{\Lambda}\right)  $ and
$u|\partial\Lambda=\psi$ then%
\[
\sup_{\Lambda}\left\vert u\right\vert \leq\sup_{\partial\Lambda}\left\vert
\psi\right\vert .
\]
\newline

\subsection{\label{bge}Boundary gradient estimates. Barriers}

\qquad In this section we obtain estimates of the norm of the gradient of a
solution at the boundary of the domain. We use the technique and several
calculations done in \cite{RT1}. We first observe that the PDE in (\ref{pdej})
is equivalent to the PDE%
\begin{equation}
Q_{J}\left[  u\right]  :=\Delta u+b\left(  \left\Vert \nabla u\right\Vert
\right)  \nabla^{2}u\left(  \frac{\nabla u}{\left\Vert \nabla u\right\Vert
},\frac{\nabla u}{\left\Vert \nabla u\right\Vert }\right)  -\left\langle
\nabla u,J\right\rangle =0. \label{PDE3}%
\end{equation}

Assume that $\psi\in C^{2,\alpha}\left(  \overline{\Lambda}\right)  $ and let
$\delta_{0}>0$ be such that $d\left(  x\right)  :=d\left(  x,\partial
\Lambda\right)  ,$ $x\in\overline{\Lambda},$ is $C^{2}$ in the strip%
\[
\overline{\Lambda}_{\delta_{0}}:=\left\{  x\in\overline{\Lambda}\text{
$\vert$
}d(x)\leq\delta_{0}\right\}  .
\]
The barrier is of the form $w=\psi+f(d)$ with $f(0)=0,$ $f\in C^{2}\left(
\left[  0,\infty\right[  \right)  .$ A computation gives%
\begin{equation}
Q_{J}\left[  w\right]  =\mathcal{L}_{w}g+f^{\prime}\mathcal{L}_{w}%
d+f^{\prime\prime}\left(  1+b\left\langle \nabla d,\frac{\nabla w}{\left\Vert
\nabla w\right\Vert }\right\rangle ^{2}\right)  \label{ele}%
\end{equation}
where $\mathcal{L}$ is the linear differential operator%
\[
\mathcal{L}_{w}v=\Delta v+b\left(  \left\Vert \nabla w\right\Vert \right)
\nabla^{2}v\left(  \frac{\nabla w}{\left\Vert \nabla w\right\Vert }%
,\frac{\nabla w}{\left\Vert \nabla w\right\Vert }\right)  -\left\langle \nabla
v,J\right\rangle .
\]

We have
\begin{align*}
\left\vert \mathcal{L}_{w}v\right\vert  &  \leq\sqrt{m}B\left\vert \nabla
^{2}v\right\vert +\left\Vert J\right\Vert \left\Vert \nabla v\right\Vert \\
&  \leq B\left(  \sqrt{m}+\left\Vert J\right\Vert \right)  \left(  \left\vert
\nabla^{2}v\right\vert +\left\Vert \nabla v\right\Vert \right)
\end{align*}
where $B=\max\left\{  1,1+b\right\}  ,$ $m=\dim N.$

Setting
\begin{equation}
c_{1}=\max_{\overline{\Lambda}_{\delta_{0}}}\left\Vert \nabla\psi\right\Vert
\label{cc1}%
\end{equation}
we have%
\begin{equation}
f^{\prime}-c_{1}\leq\left\Vert \nabla w\right\Vert \leq f^{\prime}+c_{1}
\label{c1}%
\end{equation}
and hence%
\begin{equation}
\frac{2}{3}f^{\prime}\leq\left\Vert \nabla w\right\Vert \leq\frac{4}%
{3}f^{\prime} \label{doister}%
\end{equation}
provided that
\begin{equation}
f^{\prime}\geq\alpha\geq\max\left\{  1,3c_{1}\right\}  , \label{alf}%
\end{equation}
where the number $\alpha$ will be appropriately chosen later on. We assume
$f^{\prime\prime}\leq0$ and construct a supersolution$.$

In the mild decay case we have%
\begin{gather*}
1+b\left\langle \nabla d,\frac{\nabla w}{\left\Vert \nabla w\right\Vert
}\right\rangle ^{2}\geq\left(  1+b\right)  \left\langle \nabla d,\frac{\nabla
w}{\left\Vert \nabla w\right\Vert }\right\rangle ^{2}\\
=\frac{1+b}{f^{\prime2}}\left\langle \nabla w-\nabla\psi,\frac{\nabla
w}{\left\Vert \nabla w\right\Vert }\right\rangle ^{2}\geq\frac{1+b}%
{f^{\prime2}}\left(  \left\Vert \nabla w\right\Vert -c_{1}\right)  ^{2}\\
\geq\frac{1}{4}\frac{1+b}{f^{\prime2}}\left\Vert \nabla w\right\Vert ^{2}.
\end{gather*}
This implies%
\[
\frac{4}{f^{\prime}B}Q_{J}\left[  w\right]  \leq C+\frac{f^{\prime\prime}%
}{f^{\prime3}}\frac{b+1}{B}\left\Vert \nabla w\right\Vert ^{2}%
\]
with%
\[
C=4\max_{\overline{\Omega}_{\delta}}\left(  \left(  \sqrt{m}+\left\Vert
J\right\Vert \right)  \left(  1+\left\Vert \nabla\psi\right\Vert +\left\vert
\nabla^{2}\psi\right\vert +\left\vert \nabla^{2}d\right\vert \right)  \right)
.
\]

Now, from Condition I we obtain%
\[
\frac{4}{f^{\prime}B}Q\left[  w\right]  \leq C+\frac{f^{\prime\prime}%
}{f^{\prime3}}\varphi\left(  \frac{2}{3}f^{\prime}\right)  .
\]
We may now apply exactly the same calculation done in (\cite{RT1}, pp 16 - 17)
to obtain a supersolution.

In the strong decay eigenvalue ratio case we require the condition%
\[
\Delta d-\left\langle \nabla d,J\right\rangle \leq0\text{ in }\Lambda
_{\delta_{0}}.
\]
With this condition we obtain%
\begin{align*}
f^{\prime}\mathcal{L}_{w}d  &  \leq\frac{f^{\prime}b}{\left\Vert w\right\Vert
^{2}}\nabla^{2}d\left(  \nabla\psi+f^{\prime}\nabla d,\nabla\psi+f^{\prime
}\nabla d\right) \\
&  =\frac{f^{\prime2}b}{\left\Vert w\right\Vert ^{2}}\left(  2\nabla
^{2}d\left(  \nabla\psi,\nabla d\right)  +\frac{1}{f^{\prime}}\nabla
^{2}d\left(  \nabla\psi,\nabla\psi\right)  \right) \\
&  \leq\frac{9}{4}B\left\vert \nabla^{2}d\right\vert \left(  2c_{1}+c_{1}%
^{2}\right)  \leq Bc_{0}%
\end{align*}
where $c_{1}$ is given by (\ref{cc1}) and (\ref{doister}) is supposed to hold.
It follows that%

\[
\frac{4}{B}Q_{J}\left[  w\right]  \leq4\left(  \sqrt{m}+\left\Vert
J\right\Vert \right)  \left(  \left\Vert \nabla\psi\right\Vert +\left\vert
\nabla^{2}\psi\right\vert \right)  +4c_{0}+\frac{f^{\prime\prime}}{f^{\prime
2}}\frac{1+b}{B}\left\Vert \nabla w\right\Vert ^{2}.
\]

With Condition II\ from the strong decay we finally get%
\[
\frac{4}{B}Q_{J}\left[  w\right]  \leq C+\frac{f^{\prime\prime}}{f^{\prime2}%
}\varphi\left(  \frac{4f^{\prime}}{3}\right)
\]
where the constant $C$ depends only on $m$ and
\[
\sup_{\Lambda_{\delta_{0}}}\left(  \left\Vert \nabla\psi\right\Vert
+\left\vert \nabla^{2}\psi\right\vert +\left\Vert J\right\Vert +\left\vert
\nabla^{2}d\right\vert \right)  .
\]
This leads to a supersolution as in (\cite{RT1} pp 18 - 19). We then have

\begin{proposition}
\label{bgem}Let $\Lambda$ be a bounded domain of class $C^{2}$ in $N$ and let
$\delta_{0}>0$ be such that the distance $d\left(  x,\partial\Lambda\right)
,$ $x\in\overline{\Lambda},$ is $C^{2}$ in the strip%
\[
\overline{\Lambda}_{\delta_{0}}:=\left\{  x\in\overline{\Lambda}\text{
$\vert$
}d(x)\leq\delta_{0}\right\}  .
\]
Let $u\in C^{1}\left(  \overline{\Lambda}\right)  \cap C^{2}\left(
\Lambda\right)  $ be solution of (\ref{pdej}) with $\psi\in C^{2}\left(
\overline{\Lambda}\right)  .$ We assume that either Condition I or II of
Section \ref{intro} are satisfied and in case that Condition II holds we
require furthermore the existence of $0<\delta\leq\delta_{0}$ such that the
mean curvature $H_{\partial\Lambda_{d}}$ of $\partial\Lambda_{d}$ with respect
to the interior normal vector field $\eta_{d}$ of $\partial\Lambda_{d},$
$0<d\leq\delta$ satisfies
\begin{equation}
H_{\partial\Lambda_{d}}\geq-\left\langle J,\eta_{d}\right\rangle . \label{s}%
\end{equation}
Then the normal derivative of $u$ on $\partial\Lambda$ can be estimated by a
constant depending only on $\left\vert \psi\right\vert _{C_{2}\left(
\Omega\right)  },$ $\left\vert \nabla^{2}d\right\vert $ and $\left\Vert
J\right\Vert $.
\end{proposition}

\subsection{Local and global gradient estimates\label{C}}

\qquad In this section we obtain local gradient estimates of solutions of
(\ref{PDE3}) and use Proposition \ref{bgem} to obtain global gradient
estimates of solutions of (\ref{pdej}) with smooth boundary data.

Let $u$ be a solution of (\ref{PDE3}) of $C^{3}$ class. We obtain an equation
for $\left\Vert \nabla u\right\Vert $ by differentiating (\ref{PDE3}) in
direction $\nabla u.$

\begin{lemma}
\label{pdegrad}If $u\in C^{3}\left(  \Lambda\right)  $ solves (\ref{PDE3})
then, in an orthonormal frame $E_{1},...,E_{m}$ with $E_{1}=\left\Vert \nabla
u\right\Vert ^{-1}\nabla u$ on a neighborhood of $\Lambda$ where $\nabla u$ is
non zero$,$ $m=\dim N,$ the following equality holds%
\begin{gather*}
\left(  b+1\right)  \left\Vert \nabla u\right\Vert \nabla^{2}\left\Vert \nabla
u\right\Vert \left(  E_{1},E_{1}\right)  +\left\Vert \nabla u\right\Vert
\sum_{i=2}^{m}\nabla^{2}\left\Vert \nabla u\right\Vert \left(  E_{i}%
,E_{i}\right) \\
+b^{\prime}\left\Vert \nabla u\right\Vert \nabla^{2}u\left(  E_{1}%
,E_{1}\right)  ^{2}+b\sum_{i=2}^{m}\nabla^{2}u\left(  E_{1},E_{i}\right)
^{2}-\sum_{i=1,j=2}^{m}\nabla^{2}u\left(  E_{i},E_{j}\right)  ^{2}\\
-\operatorname*{Ric}\left(  \nabla u,\nabla u\right)  -\left\Vert \nabla
u\right\Vert \nabla^{2}u\left(  E_{1},J\right)  -\left\Vert \nabla
u\right\Vert ^{2}\left\langle E_{1},\nabla_{E_{1}}J\right\rangle =0
\end{gather*}
where $\operatorname*{Ric}$ is the Ricci tensor of $N.$
\end{lemma}

\begin{proof}
We write (\ref{PDE3}) in the equivalent form%
\[
\left\Vert \nabla u\right\Vert ^{2}\Delta u+b\left(  \left\Vert \nabla
u\right\Vert \right)  \nabla^{2}u\left(  \nabla u,\nabla u\right)  -\left\Vert
\nabla u\right\Vert ^{2}\left\langle \nabla u,J\right\rangle =0,
\]
differentiate this equation in direction $\nabla u$ and afterwards divide the
result by $\left\Vert \nabla u\right\Vert ^{2}.$ Also using the relation%
\[
\nabla^{2}u\left(  \nabla u,\nabla u\right)  =\frac{1}{2}\left\langle
\nabla\left\Vert \nabla u\right\Vert ^{2},\nabla u\right\rangle
\]
and Bochner's formula%
\[
\left\langle \nabla\Delta u,\nabla u\right\rangle =\frac{1}{2}\Delta\left\Vert
\nabla u\right\Vert ^{2}-\left\vert \nabla^{2}u\right\vert ^{2}%
-\operatorname*{Ric}\left(  \nabla u,\nabla u\right)
\]
we thus get%
\begin{gather*}
\frac{1}{2}\Delta\left\Vert \nabla u\right\Vert ^{2}-\left\vert \nabla
^{2}u\right\vert ^{2}-\operatorname*{Ric}\left(  \nabla u,\nabla u\right)
+\left(  b^{\prime}\left\Vert \nabla u\right\Vert -2b\right)  \left\Vert
\nabla u\right\Vert ^{-4}\nabla^{2}u\left(  \nabla u,\nabla u\right)  ^{2}\\
+2\left\Vert \nabla u\right\Vert ^{-2}\nabla^{2}u\left(  \nabla u,\nabla
u\right)  ^{2}\left\langle \nabla u,J\right\rangle +\frac{1}{2}b\left\Vert
\nabla u\right\Vert ^{-2}\nabla u\left\langle \nabla\left\Vert \nabla
u\right\Vert ^{2},\nabla u\right\rangle \\
-\left\Vert \nabla u\right\Vert ^{-2}\nabla u\left(  \left\Vert \nabla
u\right\Vert ^{2}\left\langle \nabla u,J\right\rangle \right)  =0.
\end{gather*}

As in \cite{RT1}, Lemma 3.5, we get for the last two terms%
\[
\nabla u\left\langle \nabla\left\Vert \nabla u\right\Vert ^{2},\nabla
u\right\rangle =\left\Vert \nabla u\right\Vert ^{2}\left[  \nabla
^{2}\left\Vert \nabla u\right\Vert ^{2}\left(  E_{1},E_{1}\right)
+2\sum_{i=1}^{m}\nabla^{2}u\left(  E_{1},E_{i}\right)  ^{2}\right]
\]
and%
\begin{align*}
\nabla u\left(  \left\Vert \nabla u\right\Vert ^{2}\left\langle \nabla
u,J\right\rangle \right)   &  =\left\langle \nabla\left\Vert \nabla
u\right\Vert ,\nabla u\right\rangle \left\langle \nabla u,J\right\rangle
+\left\Vert \nabla u\right\Vert ^{2}\nabla u\left\langle \nabla
u,J\right\rangle \\
&  =2\nabla^{2}u\left(  \nabla u,\nabla u\right)  \left\langle \nabla
u,J\right\rangle \\
&  +\left\Vert \nabla u\right\Vert ^{2}\left(  \left\langle \nabla_{\nabla
u}\nabla u,J\right\rangle +\left\langle \nabla_{\nabla u}J,\nabla
u\right\rangle \right) \\
&  =\left\Vert \nabla u\right\Vert ^{2}\left[  2\nabla^{2}u\left(  E_{1}%
,E_{1}\right)  \left\langle \nabla u,J\right\rangle +\nabla^{2}u\left(  \nabla
u,J\right)  \right. \\
&  \left.  +\left\Vert \nabla u\right\Vert ^{2}\left\langle \nabla_{E_{1}%
}J,E_{1}\right\rangle \right]  .
\end{align*}

This results in
\begin{gather*}
\frac{1}{2}\left(  b+1\right)  \nabla^{2}\left\Vert \nabla u\right\Vert
^{2}\left(  E_{1},E_{1}\right)  +\frac{1}{2}\sum_{i=2}^{m}\nabla^{2}\left\Vert
\nabla u\right\Vert ^{2}\left(  E_{i},E_{i}\right) \\
+\left(  b^{\prime}\left\Vert \nabla u\right\Vert -b\right)  \nabla
^{2}u\left(  E_{1},E_{1}\right)  ^{2}+b\sum_{i=2}^{m}\nabla^{2}u\left(
E_{1},E_{i}\right)  ^{2}-\sum_{i,j=1}^{m}\nabla^{2}u\left(  E_{i}%
,E_{j}\right)  ^{2}\\
-\operatorname*{Ric}\left(  \nabla u,\nabla u\right)  -\left\Vert \nabla
u\right\Vert \nabla^{2}u\left(  E_{1},J\right)  -\left\Vert \nabla
u\right\Vert ^{2}\left\langle E_{1},\nabla_{E_{1}}J\right\rangle =0.
\end{gather*}

Using finally the relation
\[
\frac{1}{2}\nabla^{2}\left\Vert \nabla u\right\Vert ^{2}\left(  E_{i}%
,E_{i}\right)  =\left\Vert \nabla u\right\Vert \nabla^{2}\left\Vert \nabla
u\right\Vert \left(  E_{i},E_{i}\right)  +\left(  \nabla^{2}u\left(
E_{1},E_{i}\right)  \right)  ^{2},\text{ \ }1\leq i\leq m,
\]
to convert the last equation into one for $\left\Vert \nabla u\right\Vert $
instead of $\left\Vert \nabla u\right\Vert ^{2}$ we arrive at the equation in
Lemma \ref{pdegrad}.
\end{proof}

\begin{lemma}
\label{lem2} If $u\in C^{3}\left(  \Lambda\right)  $ solves (\ref{PDE3}) and
the function $G(x)=g(x)f(u)F(\left\Vert \nabla u\right\Vert )$ attains a local
maximum in an interior point $y_{0}$ of $\Lambda$ with $\nabla u(y_{0})\neq0$
then, in terms of a local orthonormal basis $E_{1}:=\left\Vert \nabla
u\right\Vert ^{-1}\nabla u,E_{2},\ldots,E_{m}$ of $T_{y_{0}}N$ we obtain, at
$y_{0},$ the relations
\begin{equation}
\frac{F^{\prime}}{F}\nabla^{2}u\left(  E_{1},E_{i}\right)  =-\frac{1}%
{g}\langle\nabla g,E_{i}\rangle-\frac{f^{\prime}}{f}\langle\nabla
u,E_{i}\rangle\label{fe}%
\end{equation}
and
\begin{align*}
0  &  \geq\left[  -\frac{F^{\prime}b^{\prime}}{F}+(b+1)\left(  \frac
{F^{\prime\prime}}{F}-\frac{F^{\prime2}}{F^{2}}\right)  \right]  \nabla
^{2}u(E_{1},E_{1})^{2}+\frac{F^{\prime}}{F\left\Vert \nabla u\right\Vert }%
\sum_{\substack{i,j\\i\geq2}}\nabla^{2}u(E_{i},E_{j})^{2}\\
&  +\left[  -\frac{F^{\prime}b}{F\left\vert \nabla u\right\vert }%
+\frac{F^{\prime\prime}}{F}-\frac{F^{\prime2}}{F^{2}}\right]  \sum_{i\geq
2}\nabla^{2}u(E_{1},E_{i})^{2}+(b+1)\left(  \frac{f^{\prime\prime}}{f}%
-\frac{f^{\prime2}}{f^{2}}\right)  \left\Vert \nabla u\right\Vert ^{2}\\
&  +\frac{\left\Vert \nabla u\right\Vert F^{\prime}}{F}\left[
\operatorname*{Ric}(E_{1},E_{1})+\langle\nabla_{E_{1}}J,E_{1}\rangle\right] \\
&  +\frac{1}{g}\left[  (b+1)\nabla^{2}g(E_{1},E_{1})+\sum_{i\geq2}\nabla
^{2}g(E_{i},E_{i})\right] \\
&  -\frac{1}{g}\left[  \sum_{i=1}^{n}\langle\nabla g,E_{i}\rangle\langle
J,E_{i}\rangle\right]  -\frac{1}{g^{2}}\left[  (b+1)\langle\nabla
g,E_{1}\rangle^{2}+\sum_{i\geq2}\langle\nabla g,E_{i}\rangle^{2}\right]  .
\end{align*}

\end{lemma}

\begin{proof}
The vanishing of
\[
\nabla\ln G=\frac{1}{g}\nabla g+\frac{f^{\prime}}{f}\nabla u+\frac{F^{\prime}%
}{F}\nabla\left\Vert \nabla u\right\Vert
\]
at $y_{0}$ gives (\ref{fe}), see \cite{RT1}, Lemma 3.6 for details. Since
(\ref{PDE3}) is elliptic and the matrix $\nabla^{2}\ln G\left(  E_{i}%
,E_{j}\right)  $ is nonpositive at $y_{0}$ the expression%
\[
\theta:=\left(  b+1\right)  \nabla^{2}\ln G\left(  E_{1},E_{1}\right)
+\sum_{i\geq2}\nabla^{2}\ln G\left(  E_{i},E_{i}\right)
\]
will be non positive at $y_{0}.$ The evaluation of $\theta$ at $y_{0}$ is now
essentially the same as in \cite{RT1}, Lemma 3.6, the only difference being
the additional term involving $J.$ This gives%
\begin{align*}
\theta &  =\frac{F^{\prime}}{F\left\Vert \nabla u\right\Vert |}\left\{
\left\Vert \nabla u\right\Vert (b+1)\nabla^{2}\left\Vert \nabla u\right\Vert
(E_{1},E_{1})+\left\Vert \nabla u\right\Vert \sum_{i\geq2}\nabla^{2}\left\Vert
\nabla u\right\Vert \left(  E_{i},E_{i}\right)  \right\} \\
&  +(b+1)\left(  \frac{F^{\prime\prime}}{F}-\frac{F^{\prime2}}{F^{2}}\right)
\nabla^{2}u(E_{1},E_{1})^{2}+\left(  \frac{F^{\prime\prime}}{F}-\frac
{F^{\prime2}}{F^{2}}\right)  \sum_{i\geq2}\nabla^{2}u(E_{1},E_{i})^{2}\\
&  +(b+1)\left(  \frac{f^{\prime\prime}}{f}-\frac{f^{\prime2}}{f^{2}}\right)
\left\Vert \nabla u\right\Vert ^{2}+\frac{1}{g}\left[  (b+1)\nabla^{2}%
g(E_{1},E_{1})+\sum_{i\geq2}\nabla^{2}g(E_{i},E_{i})\right] \\
&  -\frac{1}{g^{2}}\left[  (b+1)\langle\nabla g,E_{1}\rangle^{2}+\sum_{i\geq
2}\langle\nabla g,E_{i}\rangle^{2}\right]  +\frac{f^{\prime}}{f}\left\Vert
\nabla u\right\Vert \left\langle E_{1},J\right\rangle .
\end{align*}

We now employ Lemma \ref{pdegrad} to eliminate the second derivatives of
$\left\Vert \nabla u\right\Vert .$ Moreover, the relation%
\[
\nabla^{2}u\left(  J,E_{1}\right)  =-\frac{F}{F^{\prime}}\frac{f^{\prime}}%
{f}\left\Vert \nabla u\right\Vert \left\langle E_{1},J\right\rangle -\frac
{F}{F^{\prime}}\frac{1}{g}\sum_{i=1}^{n}\left\langle \nabla g,E_{i}%
\right\rangle \left\langle J,E_{i}\right\rangle ,
\]
which follows from (\ref{fe}), leads to a cancellation of the term
\[
\frac{f^{\prime}}{f}\left\Vert \nabla u\right\Vert \left\langle E_{1}%
,J\right\rangle
\]
in the above expression for $\theta.$ Rearranging terms immediately gives the
statement of the lemma.
\end{proof}

We now realize that Lemma \ref{lem2} coincides with Lemma 3.6 in \cite{RT1}
with only the following modifications: $\operatorname*{Ric}\left(  E_{1}%
,E_{1}\right)  $ has to be replaced by $\operatorname*{Ric}\left(  E_{1}%
,E_{1}\right)  +\left\langle \nabla_{E_{1}}J,E_{1}\right\rangle $ and the
additional term%
\[
\frac{1}{g}\sum_{i=1}^{m}\left\langle \nabla g,E_{i}\right\rangle \left\langle
J,E_{i}\right\rangle
\]
appears. These terms do not depend on $u$ and have the same structure as the
other terms involving $g.$ Hence, the complete analysis following Lemma 3.6 in
\cite{RT1} applies to the present situation and the global and local gradient
estimates for the solutions of (\ref{PDE3}) hold under the same conditions as
in \cite{RT1} (see Theorems 3.8, 3.10, 3.12 and 3.14 in \cite{RT1}). Of
course, these gradient bounds now also depend on $\left\Vert J\right\Vert $
and $\left\Vert \nabla J\right\Vert $. We will need, in particular, local
gradient estimates for the minimal surface equation in the second example of
Section \ref{last}.

\subsection{Proofs of Theorems \ref{mder1} and \ref{sder1}: conclusion}

\qquad Theorems \ref{mder1} is consequence of Proposition \ref{ab}, and the
local and global gradient estimates of Section \ref{C}. Theorem \ref{sder1}
follows from Proposition \ref{ab}, Section \ref{C} and Proposition \ref{bgem}.
As to this last proposition we note that the mean convexity in $M$ of a
$G-$invariant domain $\Omega$ is equivalent to the condition
\begin{equation}
H_{\partial\Lambda}\geq-\left\langle J,\eta\right\rangle \label{cb}%
\end{equation}
where $\Lambda=\pi\left(  \Omega\right)  .$ Indeed: at $\partial\Omega$ we
have
\[
\left\langle \nabla d,J\right\rangle \circ\pi=\left\langle \eta,J\right\rangle
\circ\pi=\left\langle \overline{\eta},\overline{J}\right\rangle =\left\langle
\nu,\overrightarrow{H}_{G}\right\rangle
\]
and the claim then follows from Lemma \ref{hahb}. As in this lemma, $\eta$
denotes the inner unit normal of $\Lambda.$

As already remarked, Theorem \ref{mder1} and Theorem \ref{sder1} are initially
proved for regular equations ($p=2$) and then carried over to the nonregular
case by the perturbation method presented in \cite{RT1}.

\section{Two examples with the minimal surface PDE\label{last}}

1) \textbf{The DP for the minimal surface equation on unbounded helicoidal
domains of} $\mathbb{R}^{3}$

\qquad Given $\lambda\geq0$ let $G_{\lambda}=\left\{  h_{t}\right\}
_{t\in\mathbb{R}}$ be the helicoidal group of isometries of $\mathbb{R}^{3}$
acting as
\[
h_{t}\left(  x,y,z\right)  =\left(  \left(
\begin{array}
[c]{cc}%
\cos\left(  \lambda t\right)  & \sin\left(  \lambda t\right) \\
-\sin\left(  \lambda t\right)  & \cos\left(  \lambda t\right)
\end{array}
\right)  \left(
\begin{array}
[c]{c}%
x\\
y
\end{array}
\right)  ,z+t\right)  ,\text{ }\left(  x,y,z\right)  \in\mathbb{R}^{3},\text{
}t\in\mathbb{R}.
\]

Let $\gamma:I\rightarrow\mathbb{R}^{2}=\left\{  z=0\right\}  $ be an arc
length curve and let $S$ be the surface generated by $\gamma$ under the action
of the helicoidal group. A parametrization of $S$ is given by
\begin{align*}
\varphi\left(  s,t\right)   &  =h_{t}\left(  \gamma\left(  s)\right)  \right)
\\
&  =\left(  x\left(  s\right)  \cos\left(  \lambda t\right)  +y\left(
s\right)  \sin\left(  \lambda t\right)  ,-x\left(  s\right)  \sin\left(
\lambda t\right)  +y\left(  s\right)  \cos\left(  \lambda t\right)  ,t\right)
\end{align*}
$t\in\mathbb{R}$, $s\in I.$ We have that $S$ is mean convex if and only if
\begin{equation}
\kappa\left(  \lambda^{2}r^{2}+1\right)  +\lambda^{2}\left(  yx^{\prime
}-xy^{\prime}\right)  \geq0 \label{e1}%
\end{equation}
where $\kappa$ is the curvature of $\gamma$ and $r=\sqrt{x^{2}+y^{2}}.$ A
sufficient condition for (\ref{e1}) is%
\[
\kappa\geq\frac{\lambda^{2}r}{\lambda^{2}r^{2}+1}.
\]

\begin{corollary}
\label{co}Consider a bounded convex $C^{2,\alpha}$ domain $\Lambda$ in
$\mathbb{R}^{2}=\left\{  z=0\right\}  \subset\mathbb{R}^{3}$ and let $\psi\in
C^{0}\left(  \partial\Lambda\right)  $ be given$.$ Set $\Omega=G_{\lambda
}\left(  \Lambda\right)  $ and let $\varphi\in C^{0}\left(  \partial
\Omega\right)  $ be defined by $\varphi=\psi\circ\pi.$ Assume that $\Omega$ is
mean convex. Then the Dirichlet problem%
\[
\left\{
\begin{array}
[c]{l}%
\operatorname{div}\frac{\nabla u}{\sqrt{1+\left\Vert \nabla u\right\Vert ^{2}%
}}=0\text{ in }\Omega\\
u|\partial\Omega=\varphi
\end{array}
\right.
\]
has a unique $G_{\lambda}-$invariant solution. A sufficient condition for the
mean convexity of $\Omega$ is
\[
\kappa\geq\frac{\lambda^{2}r}{\lambda^{2}r^{2}+1}%
\]
where $\kappa$ is the curvature of $\partial\Lambda$ and $r$ the distance
function to $\left(  0,0,0\right)  $ restricted to $\partial\Lambda.$
\end{corollary}

\textbf{2) The asymptotic DP in the hyperbolic space for the minimal surface
equation with singularities at infinity}

In the half space model%
\[
\mathbb{R}_{+}^{n}=\left\{  x\in\mathbb{R}^{n}\text{
$\vert$
}x_{n}>0\right\}  ,\text{ }g_{ij}=\delta_{ij}/x_{n}^{2},
\]
for the hyperbolic space $\mathbb{H}^{n}$ consider the one parameter subgroup
of isometries
\[
\varphi_{t}\left(  x\right)  =e^{t}x,\text{ }t\in\mathbb{R},\text{ }%
x\in\mathbb{H}^{n}.
\]
In order to compute the mean curvature vector $\overrightarrow{H}$ of the
orbits we choose the arc length parametrization%
\[
\alpha\left(  s\right)  =e^{x_{n}s}x,\text{ }s\in\mathbb{R}\text{, }\left\vert
x\right\vert =1,
\]
where $\left\vert \text{ }\right\vert $ denotes the Euclidean norm. The
Christoffel symbols with respect to the standard basis of $\mathbb{R}^{n}$ are
given by
\begin{align*}
\Gamma_{ij}^{k}  &  =0\text{ for }i<n,\text{ }j<n,\text{ }k<n\\
\Gamma_{in}^{k}  &  =-\frac{1}{x_{n}}\delta_{ik}\text{ for }k<n\\
\Gamma_{ij}^{n}  &  =\frac{1}{x_{n}}\delta_{ij}\text{ for }\left(  i,j\right)
\neq\left(  n,n\right)  ,\text{ }\Gamma_{nn}^{n}=-\frac{1}{x_{n}}.
\end{align*}

The mean curvature vector%
\[
\overrightarrow{H}=\nabla_{\frac{d}{ds}}\alpha^{\prime}(s)=\left(
H_{1},...,H_{n}\right)
\]
with%
\[
H_{k}=\alpha_{k}^{\prime\prime}+\sum_{i,j}\Gamma_{ij}^{k}\alpha_{i}^{\prime
}\alpha_{j}^{\prime}%
\]
is then computed as%
\begin{equation}
\left\{
\begin{array}
[c]{l}%
H_{k}=-x_{n}^{2}e^{x_{n}s}x_{k},\text{ }1\leq k\leq n-1\\
H_{n}=x_{n}e^{x_{n}s}\sum_{i=1}^{n-1}x_{i}^{2}.
\end{array}
\right.  \label{h}%
\end{equation}

The map%
\[
x\longmapsto\left\vert x\right\vert ^{-1}x,\text{ }x\in\mathbb{H}^{n},
\]
is seen to be a Riemannian submersion from $\mathbb{H}^{n}$ onto the $\left(
n-1\right)  -$dimensional hyperbolic space%
\[
\mathcal{S}=\left\{  x\in\mathbb{H}^{n}\text{
$\vert$
}\left\vert x\right\vert =1\right\}
\]
with its induced metric. Hence $\mathcal{S}$ is an isometric model for
$\mathbb{H}^{n}/G,$ $G=\left\{  \varphi_{t}\right\}  _{t\in\mathbb{R}},$ and,
since the orbits are orthogonal to $\mathcal{S}$ it follows that
$\overrightarrow{H}$ is tangential to $\mathcal{S}$ and hence
$J=\overrightarrow{H}|\mathcal{S}$.

One immediately sees from (\ref{h}) that $\overrightarrow{H}$ is orthogonal to
the geodesic spheres of $\mathcal{S}$ centered at $o=\left(  0,...,0,1\right)
\in\mathcal{S}$ and points towards $o.$

Note that since equivalence classes of geodesic rays (classes of
directions)\ are preserved by isometries, $G$ acts continuously on
$\partial_{\infty}\mathbb{H}^{n}.$ Moreover, it has two fixed points,
$z_{n}=+\infty$ and $\left(  0,...,0\right)  .$ Let $P:\mathbb{H}%
^{n}\rightarrow\mathcal{S}$ be the projection%
\[
P\left(  p\right)  =G\left(  p\right)  \cap\mathcal{S}%
\]
where $G\left(  p\right)  $ is the orbit through $p.$ Note that $P$ extends
continuously to a map
\[
P:\partial_{\infty}\mathbb{H}^{n}\backslash\left\{  z_{n}=+\infty,\left(
0,...,0\right)  \right\}  \rightarrow\partial_{\infty}\mathcal{S}.
\]
Given $\phi\in C^{0}\left(  \partial_{\infty}\mathcal{S}\right)  $ define
$\psi\in C^{0}\left(  \partial_{\infty}\mathbb{H}^{n}\backslash\left\{
z_{n}=+\infty,\left(  0,...,0\right)  \right\}  \right)  $ by $\psi
=\varphi\circ P.$ With these notations and remarks, we have:

\begin{corollary}
\label{h1}There is one and only one $G-$invariant solution
\[
u\in C^{\infty}\left(  \mathbb{H}^{n}\right)  \cap C^{0}\left(  \partial
_{\infty}\overline{\mathbb{H}}^{n}\backslash\left\{  z_{n}=+\infty,\left(
0,...,0\right)  \right\}  \right)
\]
of the minimal surface equation on $\mathbb{H}^{n}$ such that
\[
u|\partial_{\infty}\mathbb{H}^{n}\backslash\left\{  z_{n}=+\infty,\left(
0,...,0\right)  \right\}  =\psi.
\]

\end{corollary}

\begin{proof}
Denote also by $\varphi$ a $C^{0}$ extension of $\varphi\in C^{0}\left(
\partial_{\infty}\mathcal{S}\right)  $ to $\overline{\mathcal{S}}.$ Given
$k>0,$ from the results of Section \ref{thepde} used to prove Theorem
\ref{sder1}, we see that there is a solution $u_{k}\in C^{\infty}\left(
B_{k}\right)  \cap C^{0}\left(  \overline{B}_{k}\right)  $ of the minimal
surface equation on the geodesic ball $B_{k}$ of $\mathcal{S}$ centered at $o$
and with radius $k$ such that $u_{k}|\partial B_{k}=\varphi|\partial B_{k}$ if
inequality (\ref{cb}) is satisfied$.$ But this is the case since $H_{\partial
B_{k}}>0$ and, from what we have seen above%
\[
-\left\langle \eta,J\right\rangle =\left\langle \nabla r,\overrightarrow
{H}\right\rangle <0.
\]

From the diagonal method and the local gradient estimates (as remarked at the
end of Section \ref{C}), it follows that $u_{k}$ contains a subsequence
converging uniformly on compact subsets of $\mathcal{S}$ to a solution $u\in
C^{2}\left(  \mathcal{S}\right)  $ of the PDE (\ref{PDE3}). Regularity theory
gives $u\in C^{\infty}\left(  \mathcal{S}\right)  .$ To prove that $u\in
C^{0}\left(  \overline{\mathcal{S}}\right)  $ and $u|\partial_{\infty
}\mathcal{S}=\varphi|\partial_{\infty}\mathcal{S}$ it is enough to prove that
the PDE (\ref{pdej}) is regular at infinity (see Section 6.1 of \cite{RT1}).

Given $p\in\partial_{\infty}\mathcal{S}$ and a neighborhood $W\subset
\partial_{\infty}\mathcal{S}$ of $p,$ let $T$ be a totally geodesic
hypersurface of $\mathbb{H}^{n}$ such that $\partial_{\infty}T\subset W.$ Let
$\Omega$ be the connected component of $\mathbb{H}^{n}\backslash T$ such that
$W\subset\partial_{\infty}\Omega.$ We may assume wlg that $o\notin
\overline{\Omega}.$ We shall construct a barrier in $\Omega$ of the forma
$w=g\left(  s\right)  $ where $s:\Omega\rightarrow\mathbb{R}$ is the distance
to $T=\partial\Omega.$ Below we shall show that condition%
\begin{equation}
\left\langle \nabla s,J\right\rangle <0 \label{cnod}%
\end{equation}
is satisfied and so it turns out that $w$ will be a special case of Lemma 6.2
of \cite{RT1}. For the convenience of the reader we repeat the computation in
our special case, also leading to an explicit formula.

We may compute $\Delta s=\left(  n-2\right)  \tanh s.$ Then, for $w=g\left(
s\right)  $ with $g^{\prime}\left(  s\right)  <0$ we have%
\begin{align*}
\operatorname{div}\left(  \frac{a\left(  \left\Vert \nabla w\right\Vert
\right)  }{\left\Vert \nabla w\right\Vert }\nabla w\right)   &  =-a\left(
-g^{\prime}\left(  s\right)  \right)  \Delta s+a\left(  -g^{\prime}\left(
s\right)  \right)  g^{\prime\prime}\left(  s\right) \\
&  =-\left(  n-2\right)  a\left(  -g^{\prime}\left(  s\right)  \right)  \tanh
s+a\left(  -g^{\prime}\left(  s\right)  \right)  g^{\prime\prime}\left(
s\right)
\end{align*}
and$,$ because of (\ref{cnod}), $w$ will be a supersolution of our equation if
$g$ solves the equation%
\begin{align*}
0  &  =\tanh s-\frac{a^{\prime}}{a}\left(  -g^{\prime}\left(  s\right)
\right)  g^{\prime\prime}\left(  s\right) \\
&  =\frac{d}{ds}\left(  \left(  n-2\right)  \ln\cosh s+\ln a\left(
-g^{\prime}\left(  s\right)  \right)  \right)
\end{align*}
leading to
\[
\left(  \cosh s\right)  ^{n-2}a\left(  -g^{\prime}\left(  s\right)  \right)
=c=\text{constant.}%
\]

With
\[
a\left(  v\right)  =\frac{v}{\sqrt{1+v^{2}}}%
\]
we thus get%
\[
g\left(  s\right)  =c\int_{s}^{\infty}\frac{\left(  \cosh t\right)  ^{n-2}%
}{\sqrt{1-c^{2}\left(  \cosh t\right)  ^{4-2n}}}dt
\]
where $0<c<1$ and obviously $g\left(  0\right)  \rightarrow+\infty$ for
$c\rightarrow1.$ We show next that (\ref{cnod}) is satisfied.

Let $p\in\Omega$ and let $\Delta$ be the totally geodesic hypersurface of
$\mathbb{H}^{n}$ through $p$ orthogonal to $\nabla s$ at $p.$ Let $\alpha$ be
the level hypersurface $s^{-1}\left(  s\left(  p\right)  \right)  .$ Since
$\alpha$ is convex towards the connected component of $\mathcal{S}%
\backslash\alpha$ that contains $T$ it follows that $\Delta$ is contained in
the closure of the connected component of $\mathcal{S}\backslash\alpha$ which
does not contain $T.$ The unit normal vector $\eta$ along $\Delta$ such that
$\eta(p)=\nabla s\left(  p\right)  $ points to the connected component of
$\mathcal{S}\backslash\Delta$ which does not contain $T.$ Let $p_{0}\in\Delta$
be such that $d\left(  o,p_{0}\right)  =d\left(  o,\Delta\right)  .$ Then the
geodesic sphere centered at $o$ and passing through $p_{0}$ is tangent to
$\Delta$ at $p_{0}.$ It follows that $\left\langle \eta\left(  p_{0}\right)
,J\left(  p_{0}\right)  \right\rangle =-\left\Vert J\left(  p_{0}\right)
\right\Vert <0,$ because $J\left(  p_{0}\right)  =\overrightarrow{H}\left(
p_{0}\right)  $ and, as we have seen above, $\overrightarrow{H}\left(
p_{0}\right)  $ is orthogonal to this sphere and points to its center. We then
have that $\left\langle \eta,J\right\rangle $ is everywhere negative otherwise
it would exist a point where $J$ and $\Delta$ would be tangent. By uniqueness
of the geodesics, a geodesic of $\Delta$ would coincide with a geodesic
issuing from $o,$ contradiction! Therefore the PDE is regular at infinity and
thus $u\in C^{\infty}\left(  \mathbb{H}^{n}\right)  \cap C^{0}\left(
\partial_{\infty}\overline{\mathbb{H}}^{n}\backslash\left\{  z_{n}%
=+\infty,\left(  0,...,0\right)  \right\}  \right)  $ and
\[
u|\partial_{\infty}\mathbb{H}^{n}\backslash\left\{  z_{n}=+\infty,\left(
0,...,0\right)  \right\}  =\psi,
\]
proving the corollary.
\end{proof}

\textbf{Remarks}

a) It is clear that these examples hold for the family of PDE's (\ref{PDE1})
under the conditions of Theorems \ref{mder1} and \ref{sder1}. For simplicity
we consider only the more interesting case of the minimal surface equation.

b) The one parameter subgroup of isometries of the hyperbolic space considered
in the second example above is a particular case of transvections along a
geodesic, which is defined in any symmetric space (see \cite{H}). Thus, it
makes sense to investigate a possible extension of Corollary \ref{co} to non
compact symmetric spaces, specially on rank 1 symmetric spaces since these
have strictly negative curvature.

c) A well known problem which is being investigated in the last decades is the
existence or not of non constant bounded harmonic functions, and more recently
bounded non constant solutions of the $p-$Laplace PDE and the minimal
surface\ equation on a Hadamard manifold. A way of proving existence is by
solving the asymptotic Dirichlet problem for non constant continuous boundary
data at infinity. However, in a Hadamard manifold which is a Riemannian
product $N=M\times\mathbb{R}$, since the sectional curvature in vertical
planes are zero, it is likely true that any solution which extends
continuously to $\partial_{\infty}N$ is necessarily constant. It is a trivial
remark that bounded nonconstant $\mathbb{R}-$invariant solutions of the the
minimal surface equation which are continuous on $\overline{N}$ except for two
points in $\partial_{\infty}N$ exist when $M$ is a $2-$dimensional Hadamard
manifold with curvature bounded by above by a negative constant.

\bigskip

\begin{center}%
\begin{tabular}
[c]{lll}%
Jaime Ripoll & ** & Friedrich Tomi\\
Universidade Federal do R. G. do Sul & ** & Heidelberg University\\
Brazil & ** & Germany\\
jaime.ripoll@ufrgs.br & ** & tomi@mathi.uni-heidelberg.de
\end{tabular}

\end{center}

\end{document}